\documentclass[review]{elsarticle}

\usepackage{amsmath,amsfonts,amsthm,latexsym}
\usepackage{color}

\newcommand{\C}{\mathbb{C}}

\newcommand{\N}{\mathbb{N}}

\newcommand{\R}{\mathbb{R}}
\newcommand{\rar}{\rightarrow}
\DeclareMathOperator{\rea}{Re}

\newcommand{\T}{\mathbb{T}}

\newcommand{\Z}{\mathbb{Z}}

\DeclareMathOperator{\spec}{spec}
\DeclareMathOperator{\rank}{rank}

\newtheorem{theorem}{Theorem}[section]
\newtheorem{corr}[theorem]{Corollary}
\newtheorem{prop}[theorem]{Proposition}

\theoremstyle{definition}
\newtheorem{ex}[theorem]{Example}
\newtheorem{df}[theorem]{Definition}

\theoremstyle{remark}
\newtheorem{rem}[theorem]{Remark}

\begin{document}

\begin{frontmatter}

\title{Stability and stabilization of linear positive systems\\
on time scales}

\author{Zbigniew~Bartosiewicz}

\address{Bialystok University of Technology, Faculty of Computer Science\\ Wiejska 45A, 15-351 Bialystok, Poland}

\begin{abstract}
It is shown that a positive linear system on a time scale with a bounded graininess is uniformly exponentially stable if and only if the characteristic polynomial of the matrix defining the system has all its coefficients positive. Then this fact is used to find necessary and sufficient conditions of positive stabilizability of a positive control system on a time scale.
\end{abstract}


\end{frontmatter}

\section{Introduction}%\label{}
In positive systems the state variables take nonnegative values. Such systems appear in biology, medicine and economics \cite{FR,Ka}. We study here linear positive systems on time scales. A time scale is a model of time. Time may be continuous, discrete or mixed -- partly continuous and partly discrete. Delta derivative, which is used in delta differential equations that model the positive systems on time scales, may be equal to ordinary derivative or may be equal to a difference quotient, depending on the time scale and a particular point (see Appendix for the precise definitions).

It is known that in the continuous-time case the positive system $\dot{x}=Ax$ is exponentially stable if and only if the characteristic polynomial of the matrix $A$ has all its coefficients positive (see \cite{FR,Ka}). We show that this is true for a system on an arbitrary time scale as long as the graininess function of the time scale (which measures the distance between a particular time instant and the next instant) is bounded and uniform exponential stability is considered. Thus any discretization (in particular nonuniform) of a positive continuous-time uniformly exponentially stable system that preserves positivity gives an uniformly exponentially stable system. We rely here on \cite{DoaKalSie} and \cite{DoKaSiWi}, where uniform exponential stability of linear systems was studied. Stability of linear systems on time scales was also considered in \cite{Dav,Dac,PetRaf,Liu,DuTien,BarPioWyr,Dav1,Dav2,PoeSigWir,Ma}. Various authors considered different concepts of stability and different variants of the same concept. For example, in the definition of exponential stability, to estimate the solutions, either the standard exponential function was used or the exponential function on the time scale.

If a control system is not uniformly exponentially stable, we can try to use feedback to stabilize the system. If our system is positive, it is natural to require that the feedback preserves positivity. This procedure is called positive stabilization. We give necessary and sufficient conditions for positive stabilizability of positive systems on time scales with bounded graininess. One of the conditions is equivalent to standard stabilizability of a system on a time scale. The other consists of inequalities that must be satisfied by the matrix that defines the feedback. 

Stabilizability of positive linear systems were investigated in e.g. \cite{FR,Ka}, separately for continuous-time and discrete-time systems. Other problems for positive linear systems on time scales like reachability, observability and realizations were studied in \cite{B1,B2,B6,B8,B5,B10}.

\section{Preliminaries}
We shall study control systems on time scales.
A short introduction to the calculus on time scales is provided in Appendix. More information can be found in e.g. \cite{BohPet}.

Let $\T$ be a time scale with the forward graininess $\mu_\T$. We shall assume that $\sup\mathbb{T}=+\infty$ and denote $\mu_\T:=\sup \{\mu_\T(t): t\in\T\}$. 

 We shall  write $e_q(t,t_0)$ to denote the generalized exponential function of $q$, initialized at $t_0$ and evaluated at $t$. It is the unique forward solution ($t\geq t_0$) of the initial value problem
\[    x^\Delta(t)=q(t)x(t),\; x(t_0)=1, \]
where $x^\Delta(t)$ means the delta derivative of the function $x$ at time $t\in\T$.

For a square matrix $A$ with
complex or real elements, $\spec(A)$ will mean its spectrum,
i.e. the set of its eigenvalues and $\|A\|$ will mean the standard spectral
norm of $A$.

Let $A$ be a real ${n\times n}$ matrix. Consider the time-invariant linear system on the time scale $\mathbb{T}$
\begin{equation} \label{uk_l}
x^\Delta  (t)=Ax(t),
\end{equation}
where $x(t)\in \mathbb{R}^n$ and  $t\in\mathbb{T}$.
\begin{theorem} [\citep{BohPet}]
Let $t_0 \in \mathbb{T}$ and $x_0 \in \mathbb{R}^n$. Then
system (\ref{uk_l}) with the initial condition $x(t_0)=x_0$ has a
unique solution $x: [t_0,+\infty)\cap\mathbb{T} \rightarrow
\mathbb{R}^n$.
\end{theorem}
%\vspace{2mm}
This result can be extended to matrix-valued
solutions of \eqref{uk_l}, which leads to the following
definition.

 Let $t_0\in\mathbb{T}$. A function
\(X: [t_0,+\infty)\cap\mathbb{T} \rightarrow \mathbb{R}^{n\times n}\) that
satisfies the matrix delta differential equation
\begin{eqnarray}\label{ew1}
\mbox{$X^{\Delta}(t)=A X(t)$}
\end{eqnarray}
and the initial condition $X(t_0)=I$, where $I$ is the $n\times n$
identity matrix, is called \emph{the matrix exponential function}
(corresponding to $A$) \emph{initialized at $t_0$}. Its value at
$t\in\mathbb{T}$, $t\geq t_0$, is denoted by $e_A(t,t_0)$. Then the solution of the initial value problem
\[ x^\Delta=Ax,\  x(t_0)=x_0 \]
can be written as
\[ x(t)=e_A(t,t_0)x_0. \]

All the definitions and statements of this section can be naturally extended to complex-valued functions and matrices.

\section{Positive control systems on time scales}

Let $\R^n_+$ mean the set of all vectors from $\R^n$ with nonnegative components.

Consider now a linear control system on a time scale $\T$:

\begin{equation} \label{contsys}
x^\Delta(t)=Ax(t)+Bu(t),
\end{equation}
where $t\in\T$, $x(t)\in\R^n$ and $u(t)\in\R^m$. We assume that the control $u$ is a piecewise continuous function of time.

\begin{df}\label{d1}
The system \eqref{contsys} is \emph{positive} if for all $t_0\in\mathbb{T}$ and any control $u:[t_0,t_1]_\T\rar\R^m_+$ the trajectory starting from any $x(t_0)=x_0\in\mathbb{R}^n_+$ and corresponding to $u$ stays in $\mathbb{R}^n_+$ for all $t\in[t_0,t_1]_\T$.
\end{df}

We extend $\R$ adding $+\infty$: $\bar{\R}:=\R\cup\{+\infty\}$. For $a\in\R$ we set  $a+\infty:=+\infty$, $1/0=+\infty$, $1/+\infty=0$ and $+\infty> a$.

Let $\gamma:=\gamma(\T):={1}/{\bar{\mu}_\T}$. Then $\gamma(\R)=+\infty$, $\gamma(h\Z)=1/h$ and $\gamma(2^{\N})=0$.

Let $A_\T:=A+\gamma(\T)I$, where $\gamma(\T)I$ is the diagonal matrix with $\gamma(\T)$ on the diagonal.

\begin{prop}[\cite{B2,DoKaSiWi}]\label{ppos}
  The system $x^\Delta=Ax+Bu$ is positive if and only if all elements of $A_\T$ and $B$ are nonnegative.
\end{prop}

The condition that $A_\T$ is nonnegative implies that $A$ is a Metzler matrix, i.e. its off-diagonal elements are nonnegative. Let $M$ be an $n\times n$ Metzler matrix. Let us set $c(M):=\min\{a\geq 0 : A+aI\geq 0\}$. From the definition of $c(M)$ we get:

\begin{prop} \label{pMet}
$A_\T$ is nonnegative if and only if $A$ is Metzler and $c(A)\leq \gamma(\T)$.
\end{prop}

For a real $n\times n$ matrix $M$ let $r(M):=\max\{|\lambda|:\lambda\in\spec(M)\}$ be the \emph{spectral radius} of $M$ and $\eta(M):=\max\{\rea \lambda: \lambda\in\spec(M)\}$ be the \emph{spectral abscissa}.

\begin{theorem}[Perron-Frobenius]\label{thPF}
  Let $P\in\R^{n\times n}_+$. Then $r(P)$ is an eigenvalue of $P$.
\end{theorem}
The proof of Theorem~\ref{thPF} can be found e.g. in \cite{Ga}.

\begin{prop}[\cite{DoKaSiWi}]\label{pMet2}
Let $M$ be a real $n\times n$ Metzler matrix. Then $\eta(M)=r(M+\alpha I)-\alpha$ for all $\alpha\geq c(M)$, and $\eta(M)$ is an eigenvalue of $M$.
\end{prop}

\begin{proof}
  Since $M+\alpha I\in\R^{n\times n}_+$ for $\alpha\geq c(M)$, then Theorem~\ref{thPF} says that $r(M+\alpha I)$ is an eigenvalue of $M+\alpha I$. Thus $\eta(M+\alpha I)=r(M+\alpha I)$. But for any $\alpha\in\R$ \begin{equation}\label{e34}
    \spec(M+\alpha I)=\spec(M)+\alpha,
  \end{equation} so $\eta(M+\alpha I)=\eta(M)+\alpha$. Thus $r(M+\alpha I)=\eta(M)+\alpha$ and this gives the required equality. In particular, taking $\alpha=c(M)$ we get $\eta(M)=r(M+c(M)I)-c(M)$. Moreover, since $r(M+c(M)I)$ is an eigenvalue of $M+c(M)I$, then, by \eqref{e34}, $r(M+c(M)I)-c(M)$ is an eigenvalue of $M$.
\end{proof}

\section{Stability}

We consider here uniform exponential stability.

\begin{df}(\cite{DoaKalSie})\label{d2}
System $x^\Delta=Ax$ is \emph{uniformly exponentially stable}  if there are constants $K\geq 1$ and $\alpha>0$, and an open neighborhood $V$ of $0$ in $\mathbb{R}^n$ such that for every $t_0,t\in\mathbb{T}$ with $t\geq t_0$ and every $x_0\in\mathbb{R}^n\cap V$, the forward trajectory $x$ of the system, corresponding to the initial condition $x(t_0)=x_0$, satisfies $\|x(t)\|\leq K\exp(-\alpha(t-t_0))\|x_0\|$.
\end{df}

It is known that the condition $\sup \{\mu_\mathbb{T}(t): t\in\mathbb{T}\}<+\infty$ is necessary for uniform exponential stability of the system $x^\Delta=f(x)$ (see e.g. \cite{BP}). Therefore, from now on we shall assume that $\bar{\mu}_\T:=\sup \{\mu_\mathbb{T}(t): t\in\mathbb{T}\}<+\infty$.

For linear systems on time scales uniform exponential stability has been well investigated.
The following proposition follows directly from the definition.
\begin{prop}[\cite{DoaKalSie}]\label{stab lin}
The following conditions are equivalent:\\
 (i) system $x^\Delta=Ax$ is uniformly exponentially stable,\\
 (ii) there exist constants $\alpha>0$ and $K\geq 1$ such that for all $t_0\in\mathbb{T}$
\begin{equation}\label{ewa2}
\| e_A(t,t_0)\| \leq K
e^{-\alpha(t-t_0)}
\end{equation}
for all $t \in \mathbb{T}$ such that $t\geq t_0.$
\end{prop}

We will say that $\lambda \in \mathbb{C}$ is \emph{uniformly exponentially stable} if the scalar equation
\begin{equation} \label{r_sk}
x^\Delta=\lambda x,
\end{equation}
where $x  \in \mathbb{C}^n$,
is uniformly exponentially stable.

The set of all uniformly exponentially stable $\lambda \in \mathbb{C}$ will be denoted by $\mathcal{S}_\T$. It depends on the time scale $\T$. For $\T=\R$ it is equal to $\mathbb{C}_-$, while for $\T=h\Z$ it is the open disc of the radius $1/h$ with the center $-1/h$. But for other time scales such sets are often unknown. More details can be found in \cite{PoeSigWir,DoaKalSie}.

In \cite{DoKaSiWi} the following has been shown:

\begin{theorem} \label{thST}
  Let $\mu(\T)<+\infty$. Then $\mathcal{S}_\T\subset \mathbb{C}_-$ and for $\gamma=\gamma(\T)<+\infty$ the set $\mathcal{S}_\T$ contains the open disc of the radius $\gamma$ and the center at $-\gamma$. For $\gamma(\T)=+\infty$, $\mathcal{S}_\T=\mathbb{C}_-$.
\end{theorem}

For linear systems uniform exponential stability can be characterized by the eigenvalues of the matrix $A$.

\begin{theorem}[\citep{DoaKalSie}]\label{S2}
The following conditions are equivalent:\\
(i) system $x^\Delta=Ax$ is uniformly exponentially stable,\\
(ii) every $\lambda \in \spec (A)$ is uniformly exponentially stable.
\end{theorem}

\begin{corr}\label{c15}
  System $x^\Delta=Ax$ is uniformly exponentially stable if and only if $\spec (A)\subset \mathcal{S}_\T$.
\end{corr}

For positive systems we are interested in what happens with the trajectories that start from initial states belonging to $\R^n_+$. This leads to the following definition.

\begin{df}\label{d3}
Assume that system $x^\Delta=Ax$ is positive. We say that $x^\Delta=Ax$ is \emph{positively uniformly exponentially stable}  if there are constants $K\geq 1$ and $\alpha>0$, and an open neighborhood $V$ of $0$ in $\mathbb{R}^n$ such that for every $t_0,t\in\mathbb{T}$ with $t\geq t_0$ and every $x_0\in\mathbb{R}^n_+\cap V$, the forward trajectory $x$ of the system, corresponding to the initial condition $x(t_0)=x_0$, satisfies $\|x(t)\|\leq K\exp(-\alpha(t-t_0))\|x_0\|$ for all $t\in\mathbb{T}$ such that $t\geq t_0$.
\end{df}

It appears that for linear positive systems both properties coincide.

\begin{prop}[\cite{B12}] \label{p41}
  A positive system $x^\Delta=Ax$ is positively uniformly exponentially stable if and only if it is uniformly exponentially stable.
\end{prop}

Proposition~\ref{p41} implies that for positive linear systems positive uniform exponential stability can be characterized with the aid of the spectrum of the matrix $A$. However this spectrum for a positive system has a specific structure, so other tools for checking (positive) uniform exponential stability can be employed.

\begin{theorem}\label{th21}
  A linear positive system $x^\Delta=Ax$ is positively uniformly exponentially stable if and only if all the coefficients of the characteristic polynomial $\chi_A$ of the matrix $A$ are positive.
\end{theorem}

\begin{proof}
  Necessity. Assume that the system $x^\Delta=Ax$ is positively uniformly exponentially stable. From Proposition~\ref{p41}, Theorem~\ref{S2} and Theorem~\ref{thST} it follows that all eigenvalues $\lambda_1,\ldots,\lambda_n$ of $A$ have negative real parts. Then
  \[ \chi_A(\lambda)=(\lambda-\lambda_1)\ldots(\lambda-\lambda_n). \]
  If $\lambda_i$ is complex then $\lambda_i=-a+bi$ for $a>0$. Since $A$ is real, $\lambda_j=-a-bi$ for some $j\neq i$. Then the polynomial $(\lambda-\lambda_i)(\lambda-\lambda_j)=\lambda^2+2a\lambda+b^2$ has all its coefficients positive. If $\lambda_k$ is real, then $\lambda_k=-c$ for $c>0$, so the polynomial $\lambda-\lambda_k=\lambda+c$ has all its coefficients positive as well. Since $\chi_A(\lambda)$ is a product of polynomials of these two types, it also has all its coefficients positive.

  Sufficiency. Let us assume that $\chi_A(\lambda)=\lambda^n+a_{n-1}\lambda^{n-1}+\ldots+a_0$ has all its coefficients positive. Since $A$ is a Metzler matrix, from Proposition~\ref{pMet2} it follows that $\lambda_1:=\eta(A)\in\R$ is an eigenvalue of $A$. If it were nonnegative, then $\chi_A(\lambda_1)$ would be greater than $0$, and this would contradict the fact that $\lambda_1$ is an eigenvalue of $A$. Thus $\eta(A)<0$, which means that all eigenvalues of $A$ have negative real parts. Assume now that $\gamma(\T)<+\infty$. Since $c(A)\leq \gamma(\T)$ (from positivity of the system and Proposition~\ref{pMet}), using again Proposition~\ref{pMet2} we get
  \[  r(A+\gamma(\T)I)=\eta(A)+\gamma(\T)<\gamma(\T), \]
  which implies that $\spec(A+\gamma(\T)I)$ is contained in the disc of the radius $\gamma(\T)$ and the center at $0$. By \eqref{e34}, $\spec(A)=\spec(A+\gamma(\T)I)-\gamma(\T)$, so $\spec(A)$ is contained in in the disc of the radius $\gamma(\T)$ and the center at $-\gamma(\T)$. From Theorem~\ref{thST} we get that $\spec(A)\subset \mathcal{S}(\T)$, which, by Corollary~\ref{c15}, implies that the system is uniformly exponentially stable. For $\gamma(\T)=+\infty$, $\mathcal{S}(\T)=\C_-$, so $\spec(A)\subset \mathcal{S}(\T)$ as well.
  \end{proof}

\begin{rem}
  This fact has long been known for continuous-time systems \cite{FR,Ka}, when $\T=\R$. For a discrete-time positive system of the form
  \begin{equation}\label{e25}
    x(k+1)=\tilde{A}x(k),
  \end{equation} (uniform) exponential stability has been characterized by the condition that the characteristic polynomial of $A:=\tilde{A}-I$ has positive coefficients \cite{FR,Ka}. But \eqref{e25} is equivalent to the system on the time scale $\T=\Z$:
  \[   x^\Delta(k)=x(k+1)-x(k)=Ax(k), \]
  so this characterization of uniform exponential stability agrees with the characterization presented in Theorem~\ref{th21} for systems on arbitrary time scales with the bounded graininess. In particular, this characterization is valid for $\T=h\Z$, where $h>0$. Consider an Euler discretization of the positive continuous-time system $\dot{x}=Ax$ with the step $h$:
  \[  x^\Delta(kh):=\frac{x((k+1)h)-x(kh)}{h}=Ax(kh). \]
  Since the matrix $A$ is the same for both systems, from Theorem~\ref{th21} we conclude that uniform exponential stability of the continuous-time system implies uniform exponential stability of the discretized system, provided the latter is positive. This holds if and only if $h\leq 1/c(A)$.
\end{rem}

\section{Stabilization}

The main goal of this section is to find conditions for feedback stabilization of a positive linear system on a time scale $\T$:
\begin{equation}\label{sys}
  \Sigma: \ x^\Delta=Ax+Bu,
\end{equation}
with $x(t)\in\R^n$, $u(t)\in \R^m$.
However we want to preserve positivity of the system.

\begin{df}
System~\ref{sys} is \emph{positively stabilizable} if there is a feedback $u=Kx$, such that the closed-loop system $x^\Delta=(A+BK)x$ is positive and positively uniformly exponentially stable.
\end{df}

\begin{rem}
  We do not assume that $K$ is nonnegative, so $u=Kx$ may have negative components, even for $x\in\R^n_+$. This may be interpreted as enlarging the set of admissible control values and making it depend on the state $x$. But we still have nonnegativity of trajectories of the system that start from the points of $\R^n_+$ and correspond to controls defined by the feedback. On the other hand, if the open loop system is not positively uniformly exponentially stable, then by applying the feedback $u=Kx$ with a nonnegative $K$ we cannot stabilize the system (see \cite{Ka}).
\end{rem}

Since positive stabilizability implies usual stabilizability (where we do not require that the closed-loop system is positive), we have the following:

\begin{prop}\label{p28}
  If the positive system~\eqref{sys} is positively stabilizable, then the following condition holds:
  \begin{equation}\label{e56}
    \forall \lambda\in\spec (A):\ \lambda\notin S(\T) \Rightarrow \rank [\lambda I-A,B]=n.
  \end{equation}
\end{prop}

Condition~\eqref{e56} is equivalent to standard stabilizability \cite{BarPioWyr}, i.e. it does not guarantee existence of $K$ such that the system $x^\Delta=(A+BK)x$ is uniformly exponentially stable and positive at the same time. To find necessary and sufficient condition for positive stabilizability we need to add another condition to \eqref{e56}. For simplicity we assume now that $m=1$, so $B=b=(b_1,\ldots,b_n)^T$ and $K=(k_1,\ldots,k_n)$

Let us define for $j=1,\ldots,n$:\\
$\alpha_j:=-\infty$ if $\gamma(\T)=+\infty$ and for every $i\neq j$, $b_i=0$,\\
$\alpha_j:=\max_{i\neq j,b_i\neq 0}\{ \frac{-a_{ij}}{b_i}\}$, if $b_j=0$ or $\gamma(\T)=+\infty$ and there is $i\neq j$ such that $b_i\neq 0$,\\
$\alpha_j:= \max\{ \max_{i\neq j,b_i\neq 0}\{ \frac{-a_{ij}}{b_i}\},\frac{-a_{jj}-\gamma(\T)}{b_j}\}$ , otherwise.\\
Observe that $\alpha_j\leq 0$ for any $j=1,\ldots,n$.

Now assume that $\rank [b,Ab,\ldots,A^{n-1}b]=k$. Then $A^kb=-a_1b-\ldots -a_kA^{k-1}b$ for some $a_1,\ldots,a_k\in\R$. Let us define a basis of $\R^n$ as follows:\\
$v_{k-i}=A^ib+a_kA^{i-1}b+\ldots+a_{k-i+1}b$ for $i=0,\ldots,k-1$, and, if $k<n$, $v_{k+1},\ldots,v_n$ chosen arbitrarily so that $v_1,\ldots,v_n$ are linearly independent. In particular, $v_k=b$.

Then for $i=0,\ldots,k-2$, $Av_{k-i}=v_{k-i-1}-a_{k-i}b$ and $Av_1=a_1b$. Thus, letting $T=(v_1,\ldots,v_n)$ we get
\begin{equation}\label{e29}
  \tilde{A}:=T^{-1}AT=\begin{pmatrix}
                         0 & 1 & 0 & \cdots & 0 & * \\
                         0 & 0 & 1 & \cdots & 0 & * \\
                         \vdots &\vdots & \vdots & \ddots & \vdots & \vdots \\
                         0 & 0 & 0 & \ldots & 1 & * \\
                         -a_1 & -a_2 & -a_3 & \ldots & -a_k & * &  \\
                         0 & 0 & 0 & \ldots & 0 & A_{22}
                       \end{pmatrix}, \ \tilde{b}:=T^{-1}b=\begin{pmatrix}
                                                                      0 \\
                                                                       0\\
                                                                       \vdots \\
                                                                       0 \\
                                                                       1 \\
                                                                       0
                                                                     \end{pmatrix},
\end{equation}

where $A_{22}$ is a $(n-k)\times (n-k)$ matrix and $0$'s in the last rows of $\tilde{A}$ and $\tilde{b}$ mean zero  $(n-k)\times 1$ matrices.

Now we can state a characterization of positive stabilizability of a positive system.

\begin{theorem}
  The positive system~\eqref{sys} is positively stabilizable if and only if condition~\eqref{e56} is satisfied and the following set of linear inequalities for $K=(k_1,\ldots,k_n)$ is consistent:
  \begin{equation}\label{e67}
    k_j\geq \alpha_j,\ j=1,\ldots,n,\ Kv_i< a_i,\ i=1,\ldots,s.
  \end{equation}
 \end{theorem}

\begin{proof}
  Necessity. From Proposition~\ref{p28} it follows that condition~\eqref{e56} is necessary for positive stabilizability of system~\eqref{sys}. Moreover, positive stabilizability implies existence of $K=(k_1,\ldots,k_n)$ such that the closed-loop system $x^\Delta=(A+bK)x$ is positive. Thus such $K$ must satisfy $a_{ij}+b_ik_j\geq 0$ for $i\neq j$ and $a_{jj}+b_jk_j+\gamma(\T)\geq 0$ for $j=1,\ldots,n$. If $\gamma(\T)=+\infty$, the second inequality is always satisfied. Easy calculation shows that these inequalities are equivalent to the inequalities $k_j\geq \alpha_j$, $j=1,\ldots,n$. To show that the inequalities $Kv_i\leq a_i$ are also necessary let us transform system~\eqref{sys} to $\tilde{x}^\Delta=\tilde{A}\tilde{x}+\tilde{b}u$, where $\tilde{A}$ and $\tilde{b}$ are given by \eqref{e29}. Positive stabilizability of system~\eqref{sys} implies existence of $K$ such that $\chi_{A+bK}$ has positive coefficients. This is equivalent to existence of $\tilde{K}=KT$ such that $\chi_{\tilde{A}+\tilde{b}\tilde{K}}$ has positive coefficients (since $\tilde{A}+\tilde{b}\tilde{K}=T^{-1}(A+bK)T$ and thus the characteristic polynomials of $\tilde{A}+\tilde{b}\tilde{K}$ and $A+bK$ coincide). Observe that $\chi_{\tilde{A}+\tilde{b}\tilde{K}}(\lambda)=(\lambda^s+ (a_s-\tilde{k}_s)\lambda^{s-1}+\ldots+ (a_2-\tilde{k}_2)\lambda+(a_1-\tilde{k}_1))\chi_{A_{22}}(\lambda)$. Positivity of its coefficients implies that $\tilde{k}_i<a_i$ for $i=1,\ldots,s$. Since $\tilde{k}_i=Kv_i$, the last inequality is equivalent to $Kv_i<a_i$. Thus there must exist $K$ that satisfies \eqref{e67}, so the set of these inequalities is consistent.\\
  Sufficiency. Let $K$ satisfy \eqref{e67}. As in the proof of Necessity, the condition $k_j\geq \alpha_j$, $j=1,\ldots,n$, means positivity of the closed-loop system $x^\Delta=(A+bK)x$. To show that this system is positively uniformly exponentially stable it is enough to verify that the coefficients of $\chi_{A+bK}$ are positive. As in the proof of Necessity, this is equivalent to positivity of coefficients of $\chi_{\tilde{A}+\tilde{b}\tilde{K}}$, with $\tilde{A}$ and $\tilde{b}$ given by ~\eqref{e29} and $\tilde{K}=KT$. As before, $\chi_{\tilde{A}+\tilde{b}\tilde{K}}(\lambda)=(\lambda^s+ (a_s-\tilde{k}_s)\lambda^{s-1}+\ldots+ (a_2-\tilde{k}_2)\lambda+(a_1-\tilde{k}_1))\chi_{A_{22}}(\lambda)$ and the condition $Kv_i< a_i$, $i=1,\ldots,k$, implies $a_i-\tilde{k}_i>0$ for $i=1,\ldots,s$. Observe that the matrix $A_{22}$ defines the uncontrollable subsystem of the system~\eqref{sys}, so from ~\eqref{e56} the  eigenvalues of $A_{22}$ must belong $S(\T)$. Thus their real parts are negative, so $\chi_{A_{22}}$ has positive coefficients (see the proof of Theorem~\ref{th21}). Hence, $\chi_{\tilde{A}+\tilde{b}\tilde{K}}$ has positive coefficients as well.
\end{proof}

 \section{Conclusion}
We provided a characterization of uniform exponential stability of a positive linear system on a time scale. Surprisingly, regardless of the time scale, the condition is the same, which confirms usefulness of the theory and of the language of time scales. This characterization was used then to develop criteria for positive feedback stabilizability. They guarantee existence of a linear feedback that stabilizes the system and preserves its positivity. This was done for one-dimensional controls only, so a natural extension of this part would be to characterize positive stabilizability for the multi-control case. The next step could be studying stabilization of nonlinear systems via linearization.

%%%%%%%%%%%%%%%%%%%%%%%%
\section*{Acknowledgment}
This work has been supported by the Bialystok University of Technology grant No. S/WI/1/2016.

\section*{Appendix}
\textbf{Calculus on time scales}\\
 A  \emph{time scale} $\mathbb{T}$ is an arbitrary nonempty closed subset of the set $\mathbb{R}$
of real numbers. In particular $\mathbb{R}$, $h\mathbb{Z}$ for $h>0$ and $q^{\mathbb{N}}:=\{ q^k, k\in\mathbb{N}\}$ for $q>1$ are time scales. We assume that $\mathbb{T}$ is a
topological space with the relative topology induced from $\mathbb{R}$. If $t_0,t_1\in\mathbb{T}$, then $[t_0,t_1]_\mathbb{T}$ denotes the
intersection of the ordinary closed interval with $\mathbb{T}$. Similar notation is used for open, half-open or infinite
intervals.

For $t \in \mathbb{T}$ we define
 the {\it forward jump operator} $\sigma_\T:\mathbb{T} \rightarrow \mathbb{T}$ by
$\sigma_\T(t):=\inf\{s \in \mathbb{T}:s>t\}$ if $t\neq\sup\mathbb{T}$ and $\sigma_\T(\sup\mathbb{T})=\sup\mathbb{T}$ when $\sup\mathbb{T}$ is finite;
the {\it backward jump operator} $\rho_\T:\mathbb{T} \rightarrow \mathbb{T}$ by
$\rho_\T(t):=\sup\{s \in \mathbb{T}:s<t\}$ if $t\neq\inf\mathbb{T}$ and $\rho_\T(\inf\mathbb{T})=\inf\mathbb{T}$ when $\inf\mathbb{T}$ is finite;
  the {\it forward graininess function} $\mu_\T:\mathbb{T} \rightarrow [0,\infty)$ by
$\mu_\T(t):=\sigma_\T(t)-t$;
 the {\it backward graininess function} $\nu_\T:\mathbb{T} \rightarrow [0,\infty)$ by
$\nu_\T(t):=t-\rho_\T(t)$.

If $\sigma_\T(t)>t$, then $t$ is called {\it right-scattered}, while if $\rho_\T(t)<t$, it is called {\it
left-scattered}. If $t<\sup\mathbb{T}$ and $\sigma_\T(t)=t$ then $t$ is called {\it right-dense}. If $t>\inf\mathbb{T}$ and
$\rho_\T(t)=t$, then $t$ is {\it left-dense}.

The time scale $\mathbb{T}$ is \emph{homogeneous}, if $\mu_\T$ and $\nu_\T$ are constant functions. When $\mu_\T\equiv 0$ and $\nu_\T\equiv 0$, then $\mathbb{T}=\mathbb{R}$ or $\mathbb{T}$ is a closed interval (in particular a half-line). When $\mu_\T$ is constant and greater than $0$, then $\mathbb{T}=\mu_\T\mathbb{Z}$.

Let $\mathbb{T}^\kappa:=\{t \in \mathbb{T}: t\;\text{ is nonmaximal or left-dense}\}$. Thus $\mathbb{T}^\kappa$ is got from $\mathbb{T}$
by removing its maximal point if this point exists and is left-scattered.

Let $f:\mathbb{T} \rightarrow \mathbb{R}$ and $t \in \mathbb{T}^\kappa$.  The \emph{delta derivative of $f$ at $t$}, denoted by $f^{\Delta}(t)$, is the
real number with the property that given any $\varepsilon$ there is a neighborhood
$U=(t-\delta,t+\delta)_\mathbb{T}$  such that
\[|(f(\sigma_\T(t))-f(s))-f^{\Delta}(t)(\sigma_\T(t)-s)| \leq \varepsilon|\sigma_\T(t)-s|\]
for all $s \in U$. If $f^{\Delta}(t)$ exists, then we say that \emph{$f$ is delta differentiable at $t$}. Moreover, we say that $f$ is {\it delta differentiable} on $\mathbb{T}^k$ provided $f^{\Delta}(t)$
exists for all $t\in \mathbb{T}^k$.

\begin{ex}
If $\mathbb{T}=\mathbb{R}$, then  $f^{\Delta}(t)=f'(t)$.
 If $\mathbb{T}=h\mathbb{Z}$, then
$f^{\Delta}(t)=\frac{f(t+h)-f(t)}{h}$.
If $\mathbb{T}=q^{\mathbb{N}}$, then $f^{\Delta}(t)=\frac{f(qt)-f(t)}{(q-1)t}$.
\end{ex}

 A function
$f:\mathbb{T} \rightarrow \mathbb{R}$ is called {\it rd-continuous} provided it is continuous at right-dense points in $\mathbb{T}$ and its
left-sided limits exist (finite) at left-dense points in $\mathbb{T}$.
If $f$ is continuous, then it is rd-continuous.

A function
$f:\mathbb{T} \rightarrow \mathbb{R}$ is called \emph{regressive}, if $1+\mu(t)f(t)\neq 0$ for all $t\in\mathbb{T}$.

A function $F:\mathbb{T} \rightarrow \mathbb{R}$ is called an {\it antiderivative} of $f: \mathbb{T} \rightarrow \mathbb{R}$
provided $F^{\Delta}(t)=f(t)$ holds for all $t \in \mathbb{T}^\kappa$. Let $a,b\in\mathbb{T}$. Then the \emph{delta integral} of $f$ on the interval $[a,b)_\mathbb{T}$ is defined by
\[ \int_a^b f(\tau) \Delta \tau :=\int_{[a,b)_\mathbb{T}} f(\tau) \Delta \tau := F(b)-F(a). \]

Riemann and Lebesgue delta integrals on time scales have been also defined (see e.g. \cite{gus}).
It can be shown that every rd-continuous function has an antiderivative and its Riemann and Lebesgue integrals agree with the delta integral defined above.

\begin{ex} \label{e10}
If $\mathbb{T}=\mathbb{R}$, then $\int\limits_a^b f(\tau) \Delta \tau=\int\limits_a^b f(\tau)d\tau$, where the integral on the right is the usual Riemann integral. If $\mathbb{T}=h\mathbb{Z}$, $h>0$, then $\int\limits_a^b f(\tau)\Delta\tau=\sum\limits_{t=\frac{a}{h}}^{\frac{b}{h}-1}f(th)h$ for $a<b$.
\end{ex}

\end{document}